\documentclass[12pt]{amsart}
\usepackage{amsmath,amssymb,amsbsy,amsfonts,amsthm,latexsym,amsopn,amstext,
                            amsxtra,euscript,amscd}
\usepackage{url}

\begin{document}

\newtheorem{theorem}{Theorem}
\newtheorem{lemma}[theorem]{Lemma}
\newtheorem{algol}{Algorithm}
\newtheorem{cor}[theorem]{Corollary}
\newtheorem{prop}[theorem]{Proposition}

\newcommand{\comm}[1]{\marginpar{%
\vskip-\baselineskip 
\raggedright\footnotesize
\itshape\hrule\smallskip#1\par\smallskip\hrule}}

\def\cA{{\mathcal A}}
\def\cB{{\mathcal B}}
\def\cC{{\mathcal C}}
\def\cD{{\mathcal D}}
\def\cE{{\mathcal E}}
\def\cF{{\mathcal F}}
\def\cG{{\mathcal G}}
\def\cH{{\mathcal H}}
\def\cI{{\mathcal I}}
\def\cJ{{\mathcal J}}
\def\cK{{\mathcal K}}
\def\cL{{\mathcal L}}
\def\cM{{\mathcal M}}
\def\cN{{\mathcal N}}
\def\cO{{\mathcal O}}
\def\cP{{\mathcal P}}
\def\cQ{{\mathcal Q}}
\def\cR{{\mathcal R}}
\def\cS{{\mathcal S}}
\def\cT{{\mathcal T}}
\def\cU{{\mathcal U}}
\def\cV{{\mathcal V}}
\def\cW{{\mathcal W}}
\def\cX{{\mathcal X}}
\def\cY{{\mathcal Y}}
\def\cZ{{\mathcal Z}}

\def\C{\mathbb{C}}
\def\F{\mathbb{F}}
\def\K{\mathbb{K}}
\def\Z{\mathbb{Z}}
\def\R{\mathbb{R}}
\def\Q{\mathbb{Q}}
\def\N{\mathbb{N}}
\def\M{\textsf{M}}

\def\({\left(}
\def\){\right)}
\def\[{\left[}
\def\]{\right]}
\def\<{\langle}
\def\>{\rangle}

\def\e{e}

\def\eq{\e_q}
\def\fS{{\mathfrak S}}

\def\lcm{{\mathrm{lcm}}\,}

\def\fl#1{\left\lfloor#1\right\rfloor}
\def\rf#1{\left\lceil#1\right\rceil}
\def\mand{\qquad\mbox{and}\qquad}

\def\jt{\widetilde\jmath}
\def\ellmax{\ell_{\rm max}}
\def\llog{\log\log}
\def\disc{{\rm disc}}
\def\cl{{\rm cl}}
\def\Gal{{\rm Gal}}

%

\title[Elliptic Curves with a
Subgroup  of Prescribed Size]{Finding Elliptic Curves with 
a Subgroup of Prescribed Size}

\author{Igor E.~Shparlinski} 
\address{Department of Pure Mathematics, University of 
New South Wales, Sydney, NSW 2052, Australia}
\email{igor.shparlinski@unsw.edu.au}

\author{Andrew V. Sutherland}  
\address{Department of Mathematics, Massachusetts Institute of Technology, Cambridge, Massachusetts 02139, USA} 
\email{drew@math.mit.edu}

\subjclass{11G07, 11T06, 11Y16}
\keywords{elliptic curve, divisibility, smooth numbers, prime quadratic residues}

\begin{abstract} 
Assuming the Generalized Riemann Hypothesis, 
we design a deterministic algorithm
that, given a prime $p$ and positive integer 
$m=o(p^{1/2} (\log p)^{-4})$, 
outputs an  elliptic curve $E$ over the finite field $\F_p$ for which the
cardinality of $E(\F_p)$ is divisible by $m$.
The running time of the algorithm is $mp^{1/2+o(1)}$, and this leads
to more efficient constructions of rational
functions over $\F_p$ whose image is small relative to $p$.
We also give an unconditional version of the algorithm that works
for almost all primes $p$, and give a probabilistic algorithm with
subexponential time complexity.
\end{abstract}

\maketitle

\section{Introduction}

\subsection{Motivation} 

Let $\F_q$ denote the finite field with $q$ elements.
For an elliptic curve $E/\F_q$, we denote by $E(\F_q)$ the group of $\F_q$-rational points on $E$, which we recall is a finite abelian group; see~\cite{ACDFLNV,Gal,Silv} for background on elliptic curves and basic terminology. 
We wish to consider the problem of explicitly constructing
an elliptic curve $E/\F_q$ for which 
$$
\#E(\F_q) \equiv 0 \pmod m,
$$
for a given integer $m$. 

This problem naturally falls into the category of questions concerning the construction of
elliptic curves $E/\F_q$ for which $\#E(\F_q)$ has a prescribed arithmetic structure.
For example, motivated by cryptographic applications, many authors
have considered the problem of finding elliptic curves over finite fields for
which  $\#E(\F_q)$ is prime; 
see~\cite{ShpSuth} for an efficient probabilistic algorithm, conditional 
under the  Generalized Riemann Hypothesis (GRH).

A second motivation comes from one of the classical questions
of the theory of finite fields: constructing rational functions 
with  a small image set or, more generally, with many 
repeated values.  
It has been shown 
that results of this type are
of interest for certain cryptographic attacks; see~\cite{Cheon,ChKi1,ChKi2,Kim,KCL}, 
for example.
More precisely, for algorithms of~\cite{Cheon,ChKi1,ChKi2,Kim,KCL} 
it is important to have a polynomial or a rational function $f \in \F_q(X)$ of  prescribed degree (or with
the 
degree in a prescribed dyadic interval) such that the map $f: \F_q \to \F_q$ has many 
``collisions'', or, more formally, the equation 
$$
f(x) = f(y), \qquad x,y \in \F_q,
$$
has many off-diagonal solutions $x \ne y$. 

If $m$ divides $q-1$ then it is easy to see that the image of the
function $X^m$ has cardinality $(q-1)/m + 1$, which
is the best possible for a non-constant rational function.
If $m$ has large common divisors with both $q-1$ and $q+1$ then
{\it Dickson polynomials\/}~\cite{Brewer,Dickson1,Dickson2} of degree $m$ 
also have a reasonably small image. More precisely, 
it can be of order $q/m^{1/2}$ in the 
optimal case when $\gcd(m,q-1)$ and $\gcd(m,q+1)$ are both 
of order $m^{1/2}$, see~\cite[Theorem~9]{CGM}. But when $\gcd(m, q^2-1)=1$ neither of these constructions gives a function whose image is significantly smaller than $\F_q$. 
Rational functions with a small image set can also be constructed from {\it R{\'e}dei
functions\/}~\cite{Red},  but they require similar divisibility conditions.
Furthermore, there are several constructions
of small image polynomials in large degree extensions of finite fields, but they are usually
of degree divisible by a large power of the characteristic, 
and in any case these constructions do not work in primes fields, see~\cite{BoCo}
and references therein.

However, as observed by Cheon and Kim~\cite{ChKi2} (see also~\cite[Section~3.6]{Kim}), 
if $m$ divides $\#E(\F_q)$
then the $m$-division polynomials of an elliptic curve $E$ over $\F_q$ can be 
used to construct a suitable rational 
function $f \in \F_q(X)$. More precisely, this function is of degree  $\deg f \sim m^2$ and  maps 
  the set 
$$
\cX = \{x \in \F_q~:~ (x,y) \in E\ \text{for some}\ y \in \F_q\}
$$
of $x$-coordinates into a set  $f(\cX)$ of  cardinality  $\# f(\cX) \sim q/m$,
while $\# \cX \sim q/2$. This certainly guarantees a high number of collisions. 
A remarkable feature of this construction is that no arithmetic conditions on $q$ are required.

As a possible third motivation, we note that elliptic curves over $\F_q$ whose cardinalities are divisible by
a given integer $m$ that also divides $q-1$ play an important 
role in  the construction of Anbar and  Giulietti~\cite[Theorem~1]{AnGi}, which has applications to finite geometry and coding theory.

Here we consider the natural question of computationally efficient constructions of elliptic curves $E/\F_q$ with $\#E(\F_q)$ divisible by $m$ and design several algorithms to find such a curve.

\subsection{Notation}

Throughout the paper, the  implied constants in 
the `$O$' notation may depend, where obvious, on the real parameter $\varepsilon>0$ 
(and also on $\lambda$ in Lemma~\ref{lem:tau}), but 
are absolute otherwise.

Here we also use the `$\widetilde{O}_q$' notation to indicate that we are ignoring factors of the form $q^{o(1)}$.
That is,  for $A > 0$ we write  $\widetilde{O}_q(A)$ for a quantity bounded by $Aq^{o(1)}$.
Note that this deviates slightly from 
the more common convention that  $\widetilde{O}(A)$ indicates a 
a quantity bounded by $A (\log (A+1))^{O(1)}$.

\subsection{Naive approach}
 
 Probabilistically, for $m=o(q)$ one can easily find an elliptic curve $E/\F_q$ with $\#E(\F_q)$ divisible by $m$ in time $\widetilde{O}_q(m)$ by simply choosing curves at random.

For example, when $q$ is prime to $6$
we can simply choose random $a, b \in \F_q$ with $4a^3+27b^2\ne 0$, and then use Schoof's polynomial-time algorithm~\cite{Schoof} to determine the number of $\F_q$-rational points 
on the elliptic curve $E_{a,b}$ defined by the {\it Weierstrass equation\/}
\begin{equation}
\label{eq:Eab}
E_{a,b}:\quad Y^2  = X^3 + aX + b.
\end{equation}
If $m$ divides $\#E_{a,b}(\F_q)$ then we are done, and otherwise we may try again with another choice of $a$ and $b$.
Given that the distribution of $\#E_{a,b}(\F_q)$ over the central part of the \emph{Hasse interval}
$$
[q+1-2\sqrt{q},\ q+1+2\sqrt{q}]
$$
is not too far from uniform, we heuristically expect to find a suitable curve after $\widetilde{O}_q(m)$ such trials.

When $q$ is prime this approach can be made rigorous via
the result of Lenstra~\cite[Proposition~1.9]{Len} on the asymptotic uniformity of the number of 
Weierstrass equations that define isogenous elliptic curves;  see Lemma~\ref{lem:Len} below.
However, for large values of $m$, say $m\sim q^c$ for some $c\in (0,1)$, this algorithm is inefficient. 

Here we give more efficient solutions to a slightly modified problem.
Given some real $M\ge 1$, we seek a pair $(m,E)$ 
of an integer $m$ and 
a curve $E$ over $\F_q$ such that 
$$
\#E(\F_q) \equiv 0 \pmod m \mand m \in [M, 2M].
$$
Note that a naive approach to our modified problem involves 
\begin{itemize}
\item generating random  pairs $(a,b) \in \F_q^2$;
\item computing $\#E_{a,b}(\F_q)$;
\item factoring $\#E_{a,b}(\F_q)$;
\item checking whether  $\#E_{a,b}(\F_q)$ has a divisor $m \in [M, 2M]$.
\end{itemize}
Using probabilistic subexponential-time factoring algorithms such as those given in~\cite{LP,Pom},
this leads to an algorithm with the expected running 
time of the form $\exp\((\log q)^{1/2+o(1)}\)$. Note that 
for this approach to succeed one also has to show that 
there is a non-negligible proportion of integers~$N$ in the interval
$[q+1-\sqrt{q}, q+1-\sqrt{q}]$ (or some similar interval) 
that actually have a divisor $m\in [M,2M]$, and for which this divisor can be found efficiently.
If $N$ has many prime factors determining whether it has
such a divisor $m$ may be difficult.
This however can be achieved using an argument 
similar to that used in our proof of Theorem~\ref{thm:Card-M2M}
below. 

\subsection{Our results}

First, we use some ideas from~\cite{KoMeSh}, based on 
an algorithm of  Lenstra, Pila, and 
Pomerance~\cite{LPP1,LPP2} to show that  a more
efficient algorithm exists.  Although the ideas work for 
arbitrary finite fields, we limit ourselves to the case of 
prime $q=p$. In fact, the only missing ingredient to 
extend our result to arbitrary $q$ is a generalization of
Lenstra's result~\cite[Proposition~1.9]{Len} on the distribution of $\#E_{a,b}(\F_q)$
given in Lemma~\ref{lem:Len}
below. However, there is little doubt that this result
holds for all finite fields, so we at least
have a heuristic result in the general case.


\begin{theorem}\label{thm:Card-M2M} 
Fix any $\varepsilon> 0$.
There is a probabilistic algorithm that, given a prime $p>3$ and a real number $M$ for which $p^\varepsilon \le M \le p^{1/2 -\varepsilon}$, outputs an integer $m \in [M,2M]$ and an elliptic curve $E_{a,b}/\F_p$
for which $m \mid \#E(\F_p)$
in $\exp\((\log p)^{2/5+o(1)}\)$ expected time.
\end{theorem}

The tools used in the proof of Theorem~\ref{thm:Card-M2M}
allow us to replace the exponent $(\log p)^{2/5+o(1)}$ with a 
more precise expression involving explicit constants and double logarithms.
However, we avoid this in order to simplify the exposition and minimize
the technical details.

We also consider deterministic algorithms to solve the original problem of
constructing an elliptic curve $E/\F_p$ with $\#E(\F_p)$ divisible 
by a given integer $m$.
As a brute force method, one can modify the naive approach described above to simply
enumerate elliptic curves $E/\F_p$ (rather than generating random ones), computing $\#E(\F_p)$ in each case using Schoof's algorithm. 
But as explained in \S\ref{sec:brute} below, this yields an algorithm that 
runs in $\widetilde{O}_p(p)$ time.
Here we give a deterministic algorithm that, assuming the GRH, is more efficient than the brute force method when $m=o(p^{1/2} (\log p)^{-4})$.

We assume henceforth that $p$ always denotes a prime greater than 3.

\begin{theorem}\label{thm:Card-m GRH}
Assume the GRH.
There is a deterministic algorithm that,
given a prime $p$ and an integer $m = o(p^{1/2} (\log p)^{-4})$, 
outputs an elliptic curve $E_{a,b}/\F_p$ with $m\mid\#E(\F_p)$
in $\widetilde{O}_p(mp^{1/2})$ time. 
\end{theorem}

Furthermore, there is an unconditional algorithm that achieves the same complexity for almost all primes $p$.

\begin{theorem}\label{thm:Card-m Aver}
Let $T>1$ denote a real number and $m$ a positive integer.
For all but $O(m T^{1/2}\log T)$ primes $p \in [T,2T]$ the algorithm of Theorem~\ref{thm:Card-m GRH} outputs an elliptic curve $E_{a,b}/\F_p$ for which $m\mid\#E_{a,b}(\F_p)$
in $\widetilde{O}_p(mp^{1/2})$ time.
\end{theorem}

We note that Theorem~\ref{thm:Card-m Aver} is only interesting for $m=o(T^{1/2}/(\log T)^{2})$, since otherwise every prime $p\in[T,2T]$ may be excluded.

\section{Preparations}
\label{sec:prep}

\subsection{Isomorphism and isogeny classes of elliptic curves}
Let us fix an algebraic closure $\overline{\F}_p$ of $\F_p$.
The $\overline{\F}_p$-isomorphism class of the elliptic curve $E_{a,b}$
defined in~\eqref{eq:Eab} is uniquely determined by its $j$-\emph{invariant}
$$
j(E_{a,b}) :=  1728\frac{4a^3}{4a^3+27b^2}; 
$$
see~\cite{Silv}. 
Moreover, every $j\in\F_p$ is the $j$-invariant of some $E_{a,b}/\F_p$;
for $j\not\in \{0,1728\}$ we may take
$$
a=3j(1728-j)\quad\text{and}\quad b=2j(1728-j)^2,
$$
and for $j=0$ (resp.\ 1728) we use $a=0, b=1$ (resp.\ $a=1,b=0$).

Each $\overline{\F}_p$-isomorphism class of elliptic curves over $\F_p$ may be decomposed into a a finite number of $\F_p$-isomorphism classes.

For $j\not\in \{0,1728\}$ there are exactly two $\F_p$-isomorphism classes in the $\overline{\F}_p$-isomorphism class determined by~$j$, and they are {\it quadratic  twists\/} (meaning that they are isomorphic over $\F_{p^2}$).
For $j(E_{a,b})\not\in\{0,1728\}$ and $d\in \F_p^\times\backslash\F_p^{\times 2}$, if we set $\widetilde{a}=d^2a$ and $\widetilde{b}=d^3b$, then $E_{a,b}$ and $E_{\widetilde{a},\widetilde{b}}$ represent the two $\F_p$-isomorphism class with $j$-invariant $j(E_{a,b})=j(E_{\widetilde{a},\widetilde{b}})$.

Provided that $a\in\F_p$ is not a quadratic or cubic residue, the set
$\{E_{a^n,0}~:~n\in\Z/6\Z\}$ contains representatives for all the $\F_p$-isomorphism classes of elliptic curves with $j$-invariant 0; these $\F_p$-isomorphism classes need not be distinct, it depends on the residue class of $p\bmod 12$, but there are at most 6 of them.
Similarly, if $b\in\F_p$ is not a quadratic residue, then $\{E_{0,b^n}~:~n\in\Z/4\Z\}$ contains representatives for all the $\F_p$-isomorphism classes of elliptic curves with $j$-invariant 1728, of which there are at most 4.

It is easy to find $d\in \F_p^\times\backslash \F_p^{\times 2}$ probabilistically by applying Euler's criterion $d^{(p-1)/2}\equiv -1\bmod p$ to randomly chosen $d\in \F_p$, but one can obtain such a $d$ deterministically by simply enumerating $d\in[1,p-1]$ in order.
Under the GRH this takes $\widetilde{O}_p(1)$ time; 
the famous result of Burgess~\cite{Burg1} gives the unconditional bound $\widetilde{O}_p(p^{1/(4\sqrt{e})})$.


By a well-known theorem of Hasse, the number of $\F_p$-rational points on an elliptic curve $E/\F_p$ is of the form $p+1-t$, where $t$ is an integer with absolute value at most $2\sqrt{p}$ equal to the {\it trace of Frobenius\/}.
By a theorem of Tate, elliptic curves over a finite field have the same trace of Frobenius if and only if they are isogenous.
Thus the Hasse bound implies that there are just $O(\sqrt{p})$ distinct isogeny classes of elliptic curves over $\F_p$.

\subsection{Brute force approach}\label{sec:brute}

The most straight-forward way to construct $E/\F_p$ with $\#E(\F_p)$ divisible by $m$ is to simply enumerate pairs $(a,b)\in\F_p^2$ with $4a^2+27b^3\ne 0$ and compute $\#E_{a,b}(\F_p)$ using Schoof's algorithm~\cite{Schoof}. 
This yields an algorithm that runs in $\widetilde{O}_p(p^2)$ time, but if we instead enumerate $\F_p$-isomorphism classes, of which there are only $2p+O(1)$, we obtain an $\widetilde{O}_p(p)$ bound.
This is accomplished by enumerating $j$-invariants $j\in\F_p$ and then enumerating representatives of the (at most 6) distinct $\F_p$-isomorphism classes with the same $j$-invariant.

It is natural to suggest that an even better approach 
is possible via the  enumeration of isogeny classes, 
of which there are just $O(\sqrt{p})$.
Unfortunately we do not know an efficient way to enumerate representatives of these isogeny classes.
However, the alternative approach we propose in \S\ref{sec:alt} is able to achieve an $\widetilde{O}_p(\sqrt{p})$ running time.
In essence, we choose an isogeny class by choosing a trace of Frobenius $t\in[-2\sqrt{p},\ 2\sqrt{p}]$ for which $m$ divides $p+1-t$ and for which we can efficiently construct a representative curve $E_{a,b}/\F_p$; here we rely on the \emph{CM method}
 for constructing elliptic curves over finite fields.

\subsection{Constructing elliptic curves with the CM method} 
The theory of complex multiplication (CM) provides a standard method for constructing elliptic curves over finite fields whose group of rational points has a prescribed trace of Frobenius $t$ (and hence a prescribed number of rational points), which we now briefly recall; we refer the reader to~\cite{Cox} for additional background.

Suppose $E/\F_p$ is an elliptic curve over $\F_p$ with $\#E(\F_p)=p+1-t$,
and assume $p>3$.
If $t$ is nonzero then $E$ is an \emph{ordinary} elliptic curve, and its endomorphism ring is isomorphic to an order~$\cO$ in the ring of integers $\cO_K$ of the imaginary quadratic field $K=\Q(\sqrt{t^2-4p})$.
The elliptic curve $E$ is said to have \emph{complex multiplication} (CM) by the order $\cO$.
The prime~$p$ and the integer $t$ necessarily satisfy the \emph{norm equation}
$$
4p = t^2-v^2D,
$$
where $D$ is the discriminant of the order $\cO$.
There is a one-to-one correspondence between the set of $\overline{\F}_p$-isomorphism classes of elliptic curves $E/\F_p$ with CM by~$\cO$ and elements of the ideal class group $\cl(\cO)$; the cardinality of both sets is equal to the class number $h(D)$.

By the main theorem of complex multiplication, the ideal class group $\cl(\cO)$ is isomorphic to the Galois group $\Gal(K_\cO/K)$, where $K_\cO$ denotes the \emph{ring class field} of the order $\cO$.
The field extension $K_\cO/K$ can be explicitly constructed as $K_\cO = K(j)$, where $j$ denotes the $j$-invariant of an elliptic curve $E/\C$ with CM by $\cO$.
The minimal polynomial of $j$ over $K$ is the \emph{Hilbert class polynomial} $H_D(X)$; its degree is necessarily equal to the class number $h(D)$ and, remarkably, its coefficients lie in $\Z$ (not just in $\cO_K$).
Every root of $H_D(X)$ is the $j$-invariant of an elliptic curve $E/\C$ with CM by~$\cO$, and every elliptic curve over $\C$ with CM by~$\cO$ arises in this way.

The Deuring lifting theorem~\cite[Theorems~13.12-14]{Lang} implies that if $p$ is a prime that splits completely in $K_\cO$, equivalently, a prime satisfying the norm equation $4p=t^2-v^2D$ for some integers $t$ and $v$, then this correspondence also holds over $\F_p$.
The polynomial $H_D\in \Z[X]$ then splits completely into linear factors over $\F_p$, and its roots are precisely the $j$-invariants of the elliptic curves $E/\F_p$ that have CM by $\cO$, all of which have trace of Frobenius~$t$ and $p+1-t$ rational points.
We note that not every curve with trace of Frobenius~$t$ has CM by $\cO$, but every such curve has CM by an order in the ring of integers of the field $K=\Q(\sqrt{t^2-4p})$, and this field is uniquely determined by $p$ and $t$.
In practice one typically takes $D$ to be the discriminant of $K$ so that $\cO=\cO_K$ is the maximal order, since this minimizes $|D|$ for a given $p$ and $t$.

Thus given an integer $t$ and a prime $p$ for which $4p=t^2-v^2D$, we can construct an elliptic curve $E/\F_p$ with $\#E(\F_p)=p+1-t$ by first computing the  Hilbert class polynomial $H_D(X)$ and then finding a root $j$ of $H_D\bmod p$.
The root $j$ determines the $\overline{\F}_p$-isomorphism class of an elliptic curve $E$, and we can distinguish its $\F_p$ isomorphism class (and an explicit equation $E_{a,b}$) by checking which of a finite set of representatives $E_{a,b}$ with $j(E_{a,b})=j$ has the desired trace of Frobenius~$t$ (there are at most 6 possibilities to consider, and for $D<-4$, only $2$).
This can be done by simply computing $p+1-\#E_{a,b}(\F_p)$, but see~\cite{RS} for a more efficient method.

This method of constructing elliptic curves $E/\F_p$ with a prescribed trace of Frobenius is known as the \emph{CM method}.
Its key limitation is that when $|D|$ is large it may be infeasible to explicitly compute $H_D(X)$; the degree of $H_D$ is the class number $h(D)$, which is bounded by $O(|D|^{1/2}\log|D|)$,
see~\cite{Schur}, and the logarithm of the absolute value of its largest coefficient is $O(|D|^{1/2}(\log|D|)^2)$, see~\cite[Lemma 8]{Sut2}.
Thus the total size of $H_D$ is $O(|D|(\log|D|)^3)$ bits.
Under the GRH one can improve the logarithmic factors in all of these bounds, but in any case the best bound we have on the total size of $H_D(X)$ is $|D|^{1+o(1)}$ bits, and one heuristically expects a lower bound of the same form.
As a practical matter, the largest value of $|D|$ for which $H_D(X)$ has been explicitly computed is on the order of $10^{13}$, see~\cite{Sut2}, although there are more sophisticated methods that have made it feasible to apply the CM method to discriminants with $|D|$ as large as $10^{16}$; 
see~\cite{ES,Sut3}.

For the purposes of constructing a deterministic algorithm, we restrict ourselves to the complex analytic method of~\cite{Enge}, which is not as fast as the probabilistic algorithms used to achieve these results, but is able to achieve a time complexity of $|D|^{1+o(1)}$ without relying on randomization (or assuming the GRH); see Lemma~\ref{lem:HD}.

\subsection{An alternative approach}\label{sec:alt}

We now sketch an alternative approach to constructing an elliptic curve $E/\F_p$ with $E(\F_p)$ divisible by~$m$, using the CM method.
We enumerate isogeny classes of elliptic curves over $\F_p$ according to their trace of Frobenius $t$, and once we have found $t$ such that $p+1-t$ is divisible by $m$, we may apply the CM method to construct an elliptic curve $E/\F_p$ with trace $t$.
The time to construct $E$ with the CM method is $\widetilde{O}_p(|D|)$, where $D$ is the discriminant of the imaginary quadratic field $\Q(\sqrt{t^2-4p})$.
So long as~$m$ is not too large, there are many possible choices for $t$; in order to minimize the running time we want to choose $t$ so that $t^2-4p$ has a large square divisor, which allows us to make $|D|$ smaller.

Thus we are faced with finding an integer 
$t \in [-2\sqrt{p},\ 2\sqrt{p}]$ 
 such that $p+1-t \equiv 0 \pmod m$ and  $t^2-4p$ has a 
large square divisor~$v^2$.
Then the discriminant $D=\disc\ \Q(\sqrt{t^2-4p})=\disc\ \Q(\sqrt{(t^2-4p)/v^2})$ is relatively small in absolute 
value, allowing the Hilbert class polynomial $H_D(X)$ to be computed more quickly than in the typical case. 
In order to construct a curve in the isogeny 
class defined by $t$ we also need to find a root of $H_D(X)$, 
which has degree $h(D)=\widetilde{O}_p(|D|^{1/2})$.
This can be done in time $\widetilde{O}_p(p^{1/2}+h(D))$ using the deterministic algorithm of~\cite{BKS}; see Lemma~\ref{lem:factor}.

If $v^2$ is the largest square factor of $t^2-4p$, then the discriminant of $\Q(\sqrt{(t^2-4p)})$ is either $D=(t^2-4p)/v^2$ or $D=4(t^2-4p)/v^2$;
the latter case occurs only when $v$ is divisible by $2$, so after removing a factor of $2$ from $v$ if necessary, we may assume $t^2-4p=v^2D$.
This implies $v \mid (t^2-4p$),
and for prime $v$ this means that $p$ must be a quadratic 
residue modulo $v$. So the algorithm starts by selecting
an appropriate prime~$v$; since we require $v$ to lie
in a certain interval, this is precisely where the GRH comes 
into play. 

We now present concrete technical details. 

\section{Some Background from Analytic Number Theory}

\subsection{Bounds of character sums}

Let $\Lambda(v)$ denote the usual von~Mangoldt function defined by
$$
\Lambda(v):=
\begin{cases}
\log \ell &\quad\text{if $v$ is a power of the prime $\ell$,} \\
0&\quad\text{if $v$ is not a prime power.}
\end{cases}
$$

We start with the following bound on sums of Legendre symbols, 
which can be found in~\cite[Equation~(13.19)]{Mont}.

\begin{lemma} 
\label{lem:CharSum} Assume the GRH. For any real $U\ge 1$, we have
$$
\sum_{v \le U} \(1 - \frac{v}{U}\)\Lambda(v) \(\frac{p}{v}\) = O(U^{1/2} \log p).
$$
\end{lemma} 

Note that  the sum in Lemma~\ref{lem:CharSum} 
differs slightly from the traditional sum with the Legendre symbols
$\(v/p\)$. However, it is easy to see that $\(p/v\)$ is 
multiplicative character modulo $4p$.

\begin{cor} 
\label{cor:QuadRes} Assume the GRH. There is an
absolute constant $C> 0$ such that for 
every  prime $p$ and  real $V \ge C (\log p)^2$
there exists a prime $v \in [V, 4V]$ with 
$$
\(\frac{p}{v}\) = 1. 
$$
\end{cor} 

\begin{proof} We can certainly assume that $V <p/4$.  Suppose that 
$$
\(\frac{p}{v}\) =   -1. 
$$
for every   prime $v \in [V, 4V]$.
Then, by the prime number theorem, and partial 
summation, we easily derive
\begin{equation}
\label{eq:sum 4V}
\begin{split}
- \sum_{v  \in [V, 4V]}& \(1 - \frac{v}{4V}\)\Lambda(v) \(\frac{p}{v}\)
=   \sum_{v  \in [V, 4V]} \(1 - \frac{v}{4V}\)\Lambda(v)  \\
&=  4V -  V - \frac{1}{4V}  \sum_{v  \in [V, 4V]} v\Lambda(v) + o(V)  \\
&=  3V  - \frac{1}{4V}\(\frac{15}{2} V^2 +o(V^2)\) + o(V) =  \frac{9}{8} V + o(V).
\end{split}
\end{equation}
On the other hand, we have, trivially
\begin{equation}
\label{eq:sum V}
 \sum_{v \le V}  \(1 - \frac{v}{3V}\)\Lambda(v) \(\frac{p}{v}\) 
\le    \sum_{v \le V}   \Lambda(v) =   V + o(V)  .
\end{equation}
Hence, using~\eqref{eq:sum 4V} and~\eqref{eq:sum V}, we obtain
$$
\sum_{v \le 4V}  \(1 - \frac{v}{4V}\)\Lambda(v) \(\frac{p}{v}\)  \le  
-\frac{1}{8} V + o(V). $$

This however contradicts Lemma~\ref{lem:CharSum} (used  with $U = 4V$),
provided that  $V \ge C (\log p)^2$ for a sufficiently large 
absolute constant $C >0$.
\end{proof} 

The following statement is well-known and follows immediately from 
the P{\'o}lya-Vinogradov inequality, see~\cite[Theorem~12.5]{IwKow}.

\begin{lemma} 
\label{lem:QuadRes Aver} Let $T$ and $V$ denote real numbers for which $T >2 V> 2$.
For all but $O(T V^{-1} \log V + V\log V)$ primes $p \in [T,2T]$, 
there is a prime $v \in [V, 2V]$ with 
$$
\(\frac{p}{v}\) = 1. 
$$
\end{lemma} 

\begin{proof} Let  $\cV$ be the set of primes $v \in [V,2V]$
and let $\cP$ be the set of primes $p \in [T,2T]$
such that  
$$
\(\frac{p}{v}\) \ne 1, \qquad v \in \cV.
$$
Note that $\cP$ and $\cV$ are disjoint. Hence,
$$
\sum_{v \in \cV}\(\frac{p}{v}\) = - \# \cV,
$$
for every $p\in \cP$.
So, for the double sum
$$
W = \sum_{p \in \cP} \sum_{v \in \cV}\(\frac{p}{v}\) 
$$
we have
\begin{equation}
\label{eq:W lower}
W= - \# \cP \# \cV.
\end{equation}
On the other hand, we have
$$
|W| \le \sum_{p \in \cP} \left| \sum_{v \in \cV}\(\frac{p}{v}\) \right|.
$$
Using the Cauchy inequality and expanding the summation 
to all integers  $k \in [T,2T]$ we derive
$$
|W|^2 \le \# \cP \sum_{p \in \cP} \left| \sum_{v \in \cV}\(\frac{p}{v}\) \right|^2
\le  \# \cP \sum_{k \in [T,2T]} \left| \sum_{v \in \cV}\(\frac{k}{v}\) \right|^2.
$$
Now, squaring out and changing the order of summations, we obtain 
$$
|W|^2 
\le  \# \cP \sum_{v_1,v_2 \in \cV}\sum_{k \in [T,2T]}  \(\frac{k}{v_1v_2}\) .
$$
Finally, estimating the inner sum trivially for $v_1=v_2$ and 
using the P{\'o}lya-Vinogradov inequality for $v_1\ne v_2$
(see~\cite[Theorem~12.5]{IwKow}),
we derive
\begin{equation}
\label{eq:W upper}
|W|^2 =O\( \# \cP  \( \# \cV T +  \# \cV^2 V \log V\)\).
\end{equation}
Comparing~\eqref{eq:W lower} and~\eqref{eq:W upper} and applying the 
prime number theorem yields the desired result. 
\end{proof}

We note that by using the Burgess bound (see~\cite[Theorem~12.6]{IwKow}),
in the proof of Lemma~\ref{lem:QuadRes Aver} one can obtain 
a series of other estimates.

\subsection{Smooth numbers}

We recall that for any real $y> 1$, a positive integer~$n$ is said to be $y$-\emph{smooth} if its prime divisors are all less then or equal to~$y$. 
The {\sl Dickman--de Bruijn function} $\rho(u)$ is  defined recursively by
$$
\rho(u):=\left\{
\begin{array}{ll}
1 & \text{ if } 0 \leq u \leq 1,\\
{\displaystyle 1 - \int_{1}^{u}\frac{\rho(v-1)}{v}dv} & \text{ if } u>1.
\end{array}
\right.
$$
As usual, we denote by $\psi(x,y)$ the number of $y$-smooth $n \le x$. We need the following classical asymptotic formula for $\psi(x,y)$, which can be found in~\cite[Chapter~III.5, Corollary~9.3]{Ten}.

\begin{lemma} 
\label{lem:Smooth} For real $x \ge y > 1$ we define 
$$
u := \frac{\log x}{\log y}.
$$
For any fixed $\varepsilon > 0$, for 
$1 \le u \le \exp \((\log y)^{3/5 - \varepsilon}\)$ we have 
$$
\psi(x,y) \sim \rho(u) x   
$$
as $y\to \infty$. 
\end{lemma} 

\subsection{Arithmetic functions and smooth multiples in intervals}

Let $\tau(k)$ denote the number of positive divisors of an integer $k \ge 1$.
We need a bound on the average value of the divisor function 
$\tau(k)$ in short intervals. In particular, we use
the following special case of a  
much more general estimate of Shiu~\cite[Theorem~1]{Shiu}; 
further extensions are due to Nair and Tenenbaum~\cite{NaTe}.

\begin{lemma}\label{lem:tau}
For any fixed real  $\varepsilon, \lambda > 0$, and all sufficiently 
large real $z\ge w \ge z^\varepsilon$, 
we have
$$
\sum_{z \le k \le z+w} \tau(k)^\lambda
= O\(w(\log w)^{2^\lambda-1}\),
$$
where the implied constant depends only on $\varepsilon$ and $\lambda$.
\end{lemma}

The following statement is one of the main ingredients of the proof
of Theorem~\ref{thm:Card-M2M}. 

\begin{lemma} 
\label{lem:SmoothInt}  For any fixed $\varepsilon > 0$ and all sufficiently 
large real positive $x$, $y$ and $z$ with  
$$z^{1/2 -\varepsilon} \ge x > z^{\varepsilon}\quad \text{and}
\quad \exp((\log x)^{1-\varepsilon}) 
\ge y \ge \exp\((\log \log x)^{5/3 + \varepsilon}\),
$$
define $u$ by $y^u=x$. There are at least $z^{1/2} u^{-u + o(u)} (\log z)^{-3}$ integers 
$k \in  [z, z+z^{1/2}]$ that have a $y$-smooth divisor $m \in [x, 2x]$
and for which $\tau(k) \le  u^{u + o(u)}(\log z)^{2}$.
\end{lemma} 

\begin{proof} Let $\cM$ be the set of  $y$-smooth integers $m \in [x, 2x]$.
It follows from Lemma~\ref{lem:Smooth} and well known results on the growth 
of $\rho(u)$ (see~\cite[Section~III.5.4]{Ten}),  that 
\begin{equation}
\label{eq:set M}
 \# \cM = u^{-u + o(u)} x
\end{equation}
as $u \to \infty$.

For each $m\in \cM$ we consider the products $k = mr$ where 
$r$ runs through $z^{1/2}/m + O(1)$ integers of the interval $[z/m, z/m+z^{1/2}/m]$.
Let $\vartheta(k)$ be the number of such representations.
Clearly,
\begin{align*}
\sum_{k \in [z, z+z^{1/2}]} \vartheta(k) &\ge \sum_{m \in \cM} \(z^{1/2}/m + O(1)\) \\
&= (1+o(1)) z^{1/2} \sum_{m \in \cM} 1/m \ge (1/2+o(1)) z^{1/2} \# \cM x^{-1}.
\end{align*}
Hence, using~\eqref{eq:set M}, we derive
\begin{equation}
\label{eq:theta 1}
\sum_{k \in [z, z+z^{1/2}]} \vartheta(k) \ge  
 z^{1/2}u^{-u + o(u)}
\end{equation}
as $u \to \infty$.

On the other hand, since we obviously have $\vartheta(k) \le \tau(k)$,
we obtain from Lemma~\ref{lem:tau} with $\lambda = 2$ the bound
\begin{equation}
\label{eq:theta 2}
\sum_{k \in [z, z+z^{1/2}]} \vartheta(k)^2 = O\(z^{1/2}(\log z)^{3}\).
\end{equation}
Thus if $\cK$ is the set of $k\in    [z, z+z^{1/2}]$ 
with $\vartheta(k) > 0$, then by the Cauchy inequality we have
$$
\(\sum_{k \in [z, z+z^{1/2}]} \vartheta(k)\)^2
\le \# \cK \sum_{k \in [z, z+z^{1/2}]} \vartheta(k)^2.
$$
Using~\eqref{eq:theta 1} and~\eqref{eq:theta 2}, we then derive
$$
\# \cK \ge z^{1/2}u^{-u + o(u)} (\log z)^{-3}. 
$$

Now $\cE$ be the set of $k\in    [z, z+z^{1/2}]$ with 
$\tau(k) >  u^{u + o(u)}(\log z)^{3}$. Using
Lemma~\ref{lem:tau} with $\lambda = 2$ again, we obtain
$$
\# \cE \(u^{u + o(u)}(\log z)^{3}\)^2 =  O\(z^{1/2}(\log z)^{3}\).
$$
Hence 
$$
\# \cE \le  z^{1/2}u^{-2u + o(u)} (\log z)^{-3} = o(\# \cK),
$$
which concludes the proof. 
\end{proof} 

\subsection{Class numbers and the distribution of
the number of $\F_p$-rational 
points on elliptic curves}

Finally, we require a result of Lenstra~\cite[Proposition~1.9]{Len} that relates the number of elliptic curves $E_{a,b}/\F_p$
with trace of Frobenius $t$ to the Hurwitz-Kronecker class number $H(t^2-4p)$. 
Here we formulate this result in a form convenient for our applications.

\begin{lemma}
\label{lem:Len} For any set of integers $\cS \in [p-p^{1/2}, p+p^{1/2}]$
of cardinality $\# \cS \ge 3$, we have
$$
\#\{(a,b)\in \F_p^2~:~\#E_{a,b}(\F_p)\in \cS\} \gg \# \cS p^{3/2} /\log p .
$$
\end{lemma}

\section{Some Background on Algorithms}

\subsection{Finding smooth factors of integers}

The following is a simplified version of a result of 
Lenstra, Pila, and 
Pomerance~\cite[Theorem~1.1]{LPP1}, which gives a 
slower but rigorous version of the {\it elliptic curve factorisation
method\/} (ECM) of Lenstra~\cite{Len}.

\begin{lemma}\label{lem:SmoothFact}
There is a probabilistic algorithm that, given an 
integer~$n$ and a real number $y>2$, finds all prime factors
$\ell\le y$ of $n$ in expected time 
$\exp\((\log y)^{2/3 + o(1)}\) (\log n)^{O(1)},$ as $y \to \infty$. 
\end{lemma}

\subsection{Finding roots of polynomials}

We also need the following factorisation algorithm from~\cite{BKS}.

\begin{lemma}\label{lem:factor}
There is a deterministic algorithm that,
given a square-free polynomial $f \in \F_p[X]$ of degree $d$
that splits completely into linear factors in $\F_p[X]$,
finds a root of $f$ in $\widetilde{O}_p(d+p^{1/2})$ time.
\end{lemma}

The algorithm of Lemma~\ref{lem:factor} improves that 
of Shoup~\cite{Shoup} when $d$ grows as a power of $p$, which 
is exactly the case we need.

\subsection{Counting rational points on elliptic curves over finite fields}

We recall the classical result of Schoof~\cite{Schoof},
which is quite sufficient for our purposes (although we use it 
only for prime fields $\F_p$, we state it in full generality). 

\begin{lemma}
\label{lem:Schoof}
There is a deterministic algorithm that,
given an elliptic curve  $E/\F_q$, outputs
the cardinality $N = \#E(\F_q)$ in $(\log q)^{O(1)}$ time.
\end{lemma}

\subsection{Computing Hilbert class polynomials}

For computing Hilbert class polynomials deterministically, we rely on the complex analytic approach of Enge~\cite{Enge}, which uses floating point approximations of complex numbers, combined with a rigorous bound on the precision needed to control rounding errors due to Streng; see~\cite[Remark~1.1]{Streng}.

\begin{lemma}
\label{lem:HD}
There is a deterministic algorithm that,
given an imaginary quadratic discriminant $D$,
outputs $H_D(x)$ in $|D|^{1+o(1)}$ time.
\end{lemma}

\section{Proofs of Main Results}

\subsection{Proof of Theorem~\ref{thm:Card-M2M}}
Let
$$
y= \exp\left((\log p)^{3/5}\right).
$$
We choose a pair $(a,b)\in \F_p^2$ uniformly at random,
and if $4a^3+27b^2\ne 0$, we
compute the cardinality
$N = \#E_{a,b}(\F_p)$ in $(\log p)^{O(1)}$ time, via
Lemma~\ref{lem:Schoof}. 
We then use the probabilistic algorithm of 
Lemma~\ref{lem:SmoothFact} to find all the
prime divisors $\ell \le y$ of $N$ in
$$
T_1 = \exp\((\log y)^{2/3 + o(1)}\) (\log p)^{O(1)}
$$
expected time, and we can easily determine the largest power
of each of the primes $\ell$ that divides $N$ within the same time bound,
using repeated divisions by $\ell$.

One can check that for the above choice of $y$ the conditions 
of Lemma~\ref{lem:SmoothInt}  are satisfied with $x=M$ and $z = p$.
Hence, by Lemmas~\ref{lem:SmoothInt} and~\ref{lem:Len}, after an expected
$$
T_2 = u^{u + o(u)} (\log p)^{4} = \exp\( u^{1+o(1)}\)
$$ 
random 
choices of pairs $(a,b)\in \F_p^2$, where 
$$
u = \frac{\log M}{\log y},
$$ 
we find a pair $(a,b) \in \F_p^2$ for which 
$N = \#E_{a,b}(\F_p)$ has a $y$-smooth factor $m \in [M,2M]$ 
and also has $\tau(k) \le u^{u + o(u)} (\log p)^{3}$
integer divisors. 
By exhaustively checking every $y$-smooth divisor of $N$
(constructed as products of powers of prime divisors $\ell\le y$ of $N$),
for any given $N$ we can deterministically find such an $m$ 
(or determine that none exists) in time 
$$
T_3 = u^{u + o(u)} (\log p)^{O(1)} = \exp\( u^{1+o(1)}\) .
$$

This leads to a total expected running time of
$$
T_1T_2T_3 
= \exp\((\log y)^{2/3 + o(1)} + u^{1+o(1)}\) (\log p)^{O(1)}.
$$
Recalling the choice of $y$, we conclude the proof. $\qed$

\subsection{Proof of Theorem~\ref{thm:Card-m GRH}}

We let $V=p^{1/4}/(2m^{1/2})$. 
Since,  $m =o\(p^{1/2} (\log p)^{-4}\))$,
we see that if $p$ is  sufficiently large, then $V$ satisfies the 
condition of   Corollary~\ref{cor:QuadRes}.
 Combining Corollary~\ref{cor:QuadRes} 
with the deterministic primality 
test of~\cite{AKS}, we see that in time $Vp^{o(1)}$ 
we can find a prime $v\in [V, 2V]$ for which $p$ is a 
quadratic residue.  Thus the congruence $4p\equiv x^2 \pmod v$ 
has  a solution that can also be found in time  $Vp^{o(1)}$ 
using brute force search. Via Hensel lifting, we can 
now find a solution $s$ to the congruence
$$4p\equiv x^2 \pmod {v^2},$$
with $0 \le s \le v^2-1$, in time $Vp^{o(1)}$; see~\cite{vzGG}. 

Any admissible value of $t$ must satisfy the 
congruences 
$$
t \equiv s \pmod {v^2} \mand t \equiv p+1 \pmod {m}.
$$
Using the Chinese remainder theorem, in time $p^{o(1)}$
we can find an integer~$a$ with
$0 \le a \le mv^2-1$, such that the above system of congruences is
equivalent to the single congruence $t \equiv a \pmod {mv^2}$. 
Since $mv^2 \le 16mV^2 = 4p^{1/2}$,
 there is a
$t \in [-2p^{1/2}, 2p^{1/2}]$ that satisfies this congruence
(either $a$ or $a-mv^2$ must lie in the desired interval).

We now bound the complexity of constructing an elliptic
curve $E/\F_p$ with $\# E(\F_p)=p+1-t$ for our chosen value of $t$.

Let us write $t^2-4p = u^2D$, for an integer $u$
and a fundamental discriminant $D< 0$.
Then $u \ge v \ge V$, and therefore $|D| \le 4p/V^2$. 
By Lemma~\ref{lem:HD}, we can construct the
Hilbert class polynomial $H_D(x)$ of 
degree
$$
h(D) = \widetilde{O}_p(|D|^{1/2}) = \widetilde{O}_p(m^{1/2}p^{1/4}) 
$$
in time $\widetilde{O}_p(|D|) = \widetilde{O}_p(mp^{1/2})$. The result
now follows from Lemma~\ref{lem:factor}. $\qed$

\subsection{Proof of Theorem~\ref{thm:Card-m Aver}}
We can assume that $m = o(T^{1/2} (\log T)^{-2})$ as otherwise 
the bound is trivial. 
We set $V = T^{1/4}/m^{1/2}$ and discard 
$$
O(T V^{-1} \log V + V\log V) = O(mT^{1/2} \log T)
$$
primes $p\le T$, as described in  Lemma~\ref{lem:QuadRes Aver}. 

For each of the remaining primes we can find a prime $v \in [V, 2V]$ with 
$$
\(\frac{p}{v}\) = 1. 
$$
We also note that $mv^2 \le 4mV^2 = 4T^{1/2} \le 4p^{1/2}$
for every prime $p \in [T,2T]$. 
After this the proof is identical to that of Theorem~\ref{thm:Card-m GRH}.  $\qed$

\section{Possible Extensions and Generalisations}

We may also consider a heuristic version of
Theorem~\ref{thm:Card-M2M} which uses the elliptic 
curve factorisation method of Lenstra~\cite{Len}. 
This leads to the heuristic complexity bound
\begin{equation*}
\begin{split}
\exp\((\log y)^{1/2 + o(1)}\) & u^{2u + o(u)}(\log p)^{O(1)}\\
&= \exp\((\log y)^{1/2 + o(1)} + u^{1+o(1)}\) (\log p)^{O(1)},
\end{split}
\end{equation*}
which after the choice 
$$
y= \exp((\log p)^{2/3}))
$$
leads to roughly the same expected running time $\exp\((\log p)^{1/3+o(1)}\)$
as the \emph{number field sieve}, the heuristically fastest integer factorisation algorithm; see~\cite{CrPom} for more details.   

As an analog of Theorem~\ref{thm:Card-M2M}, one can also consider the case where the integer $m$ is fixed and the prime $p$ is allowed to vary over an interval $[P,2P]$, for some real $P > m^{1+\varepsilon}$ and a fixed $\varepsilon > 0$. 
If we pick a multiple $N$ of $m$ that lies in the interval
$$
[P+1+2\sqrt{P},\ 2P+1-2\sqrt{2P}],
$$ we can apply the algorithm of Br\"oker and Stevenhagen~\cite{BrSt} to construct an elliptic curve $E/\F_p$ for which $\#E(\F_p)=N$ is a multiple of~$m$; the bounds on $N$ ensure that $p\in [P,2P]$.  The heuristic expected running time of this probablistic algorithm is
$(2^{\omega(N)}\log N)^{O(1)}$, where $\omega(N)$ denotes the number of distinct prime divisors of $N$.  We have a fair amount of freedom in the choice of $N$ and can easily choose~$N$ so that we have $\omega(N)=\omega(m)+1$; this allows us to write the time bound as $(2^{\omega(m)}\log P)^{O(1)}$.
For almost all integers $m$ we have $\omega(m)=O(\log \log P)$, in which case we obtain a heuristic polynomial-time algorithm.

The algorithms of Theorems~\ref{thm:Card-m GRH} 
and~\ref{thm:Card-m Aver} can easily be extended to produce elliptic curves 
$E$ with $\# E(\F_p)$ in a given residue class modulo~$m$.

Finally, we note that our approach can be used to construct elliptic 
curves $E$ over $\F_p$ for which the group $E(\F_p)$ contains a 
prescribed subgroup. By a classical result of Waterhouse~\cite{Wa},
for a curve $E$ over $\F_p$ with $N=\#E(\F_p)$ 
all subgroups of $E(\F_p)$ are isomorphic to 
subgroups of the form 
$$
\cG_{r,s} := (\Z/r\Z) \times (\Z/rs\Z)
$$
for some positive integers $r$ and $s$ with 
$$
r\mid p-1 \mand r^2s  \mid  N.
$$
Furthermore, by results of R\"uck~\cite{Ru} and Voloch~\cite{Vol},
the above divisibilities and the condition 
$N\in [p+1-2\sqrt{p},\ p+1+2\sqrt{p}]$ are sufficient 
for the existence of a curve $E$ over $\F_p$ with 
$\cG_{r,s} \subseteq E(\F_p)$ and $\#E(\F_p)=N$, provided that $rs$ is not divisible by $p$.
This last requirement certainly holds if we require $m=r^2s$ to be less than $p$, which we do, since we are only considering $m=o(p^{1/2}/(\log p)^2)$.

To design an efficient algorithm to construct such curves, one first
obtains an analogue of Lemma~\ref{lem:QuadRes Aver} with $m = r^2s$ 
and primes $p\in [T, 2T]$ from the arithmetic progression 
$p \equiv 1 \pmod r$, which involves the same analytic tool
combined with results about primes in arithmetic 
progressions.
Given a prime $p\equiv 1\pmod r$ and a nonzero integer $t$
with $N=p+1-t$ divisible by $m=r^2s$ such that $4p=t^2-v^2D$ with
$D=\disc(\Q(\sqrt{t^2-4p})$,
we then proceed as before.  We compute the Hilbert class polynomial $H_D$,
find a root $j$ of $H_D\in \F_p[x]$, and determine a curve $E_{a,b}/\F_p$ that
has this $j$-invariant and is in the correct $\F_p$-isomorphism class so
that the trace of its Frobenius endomorphism $\pi$ is equal to $t$.
Then $E_{a,b}(\F_p)$ contains a subgroup isomorphic to the prescribed group $\cG_{r,s}$, as we now argue.

Since $D$ is the discriminant of $K=\Q(\sqrt{t^2-4p})$, the endomorphism ring of~$\mathrm{End}(E_{a,b})$ is
isomorphic to the maximal order $\cO_K$.
By applying \cite[Lemma 2]{Ru}, we can write $\pi-1=\ell^a\omega$,
with $\omega\in\cO_K=\mathrm{End}(E_{a,b})$ and
$$
a = \left\{v_\ell(p-1),v_\ell(N)/2\right\}.
$$
It then follows from \cite[Lemma 1]{Ru} that $E_{a,b}(\F_p)$ contains a
subgroup isomorphic to
$$
\Z/\ell^a\Z\times \Z/\ell^b\Z,
$$
where $b=v_\ell(N)-a$.
By construction, we have $a\ge v_\ell(r)$ and $b\ge v_\ell(rs)$,
and it follows that $E_{a,b}(\F_p)$ contains a subgroup isomorphic to the
$\ell$-Sylow subgroup of $\cG_{r,s}$.
Since this holds for all primes $\ell$, we see that
$E_{a,b}(\F_p)$ contains a subgroup isomorphic to $\cG_{r,s}$ as claimed.

\section{Some Facts About Primes in Arithmetic Progressions}

We now present several facts that shed some light on the 
frequency of pairs $(m,p)$ with $p \equiv \pm 1 \pmod m$,
which is important for better understanding the the advantages
of using elliptic curves for constructing polynomial maps with 
many collisions. 

For any fixed $m$ this is essentially a 
result about the distribution of primes in 
arithmetic progressions. In particular, the standard proof 
of Linnik's theorem on the smallest prime in an arithmetic 
progression implies that there is an absolute constant 
$K> 0$ such that  for any integer $m\ge 2$, for all $T \ge m^K$ 
there exists a prime $p\in [T,2T]$ in any admissible residue
class modulo $m$; see~\cite[Theorem~18.6]{IwKow}. 
It would be interesting to see what 
the currently strongest  approaches to estimates 
of the {\it Linnink constant\/} $L$ of Heath-Brown~\cite{H-B}
(with $L \le 5.5$), and of 
T. Xylouris~\cite{Xyl} (with $L \le 5.18$), give for the
above constant $K$.

We also note that, by a result of Mikawa~\cite{Mik}, for any
sufficiently large $M$, for all but $o(M)$ integers 
$m \in [M,2M]$, for any $K > 32/17$ and $T > M^K$  
there exists a prime $p\in [T,2T]$ with  $p\equiv 1 \pmod m$ 
(and also with $p\equiv -1 \pmod m$). The classical 
Bombieri-Vinogradov Theorem~\cite[Theorem~17.1]{IwKow}
gives similar results for $K > 2$. 

Finally, we note that several results of  Ford~\cite{Ford}
can also provide some information on the existence and distribution
of pairs $(m,p)$ with $p \equiv \pm 1 \pmod m$. 
For example, a combination of~\cite[Corollary~2]{Ford}
and~\cite[Theorem~6]{Ford} implies that as both $M$ and $T/M$ tend to infinity,
there are only $o(T/\log T)$ primes $p \in [T,2T]$ such that 
$p-1$ has a divisor $m \in [M,2M]$. On the other hand, by a 
slight modification of~\cite[Theorem~7]{Ford}, 
for any $\beta> \alpha>0$, there are at least $cT/\log T$
primes  $p \in [T,2T]$ such that 
$p-1$ has a divisor $m \in [T^\alpha,T^\beta]$.



\section*{Acknowledgement}

The authors would like to thank Jung-Hee Cheon
for very useful comments and Michael Zieve for providing
precise 
information about value sets of Dickson polynomials. 
The authors are also very grateful to the referee 
for a careful reading of the manuscript. 

During the preparation of this paper  I.~E.~Shparlinski was supported in part
by ARC grant DP130100237 and
A.~V.~Sutherland received financial support from NSF grant DMS-1115455.

\end{document}